\pdfoutput=1
\RequirePackage{ifpdf}
\ifpdf % We~are running pdfTeX in pdf mode
\documentclass[pdftex]{sigma}
\else
\documentclass{sigma}
\fi

\usepackage{amscd,eucal}

\usepackage{tikz}
\usetikzlibrary{lindenmayersystems}

\newtheorem{thm}{Theorem}[section]

\newtheorem{lem}[thm]{Lemma}

\newtheorem{prop}[thm]{Proposition}

\theoremstyle{definition}
\newtheorem{dfn}[thm]{Definition}
\newtheorem{rem}[thm]{Remark}

\numberwithin{equation}{section}

\def\ca{{\mathcal A}}
\def\cb{{\mathcal B}}
\def\cc{{\mathcal C}}
\def\cd{{\mathcal D}}

\def\cg{{\mathcal G}}
\def\ch{{\mathcal H}}
\newcommand{\cai}{\mathcal{I}}
\def\cj{{\mathcal J}}

\def\cl{{\mathcal L}}
\def\cam{{\mathcal M}}

\def\cs{{\mathcal S}}

\def\cu{{\mathcal U}}

\newcommand{\bn}{\mathbb N}
\def\br{{\mathbb R}}

\def\bz{{\mathbb Z}}

\renewcommand{\a}{\alpha}

\def\g{\gamma} 
\def\d{\delta} 
\def\eps{\varepsilon}

\def\l{\lambda} 
\def\m{\mu}

\def\p{\pi}

\def\s{\sigma} 
 
\def\t{\tau}

\def\f{\varphi}

\def\om{\omega} \def\O{\Omega}

\DeclareMathOperator{\tr}{tr}
\DeclareMathOperator{\Res}{Res}
\DeclareMathOperator{\vol}{\mu_d}

\DeclareMathOperator{\Lim}{Lim}

\def\length{\mathrm{length}}
\def\size{\mathrm{size}}
\def\level{\mathrm{level}}
\newcommand{\dg}{d_{\rm geo}}
\newcommand{\Balg}{\cb_{\text{fin}}}

\newcommand{\ps}[2]{( #1 , #2 )}

\begin{document}
%\allowdisplaybreaks

\newcommand{\arXivNumber}{2005.14225}

\renewcommand{\thefootnote}{}

\renewcommand{\PaperNumber}{020}

\FirstPageHeading

\ShortArticleName{A Spectral Triple for a Solenoid Based on the Sierpinski Gasket}

\ArticleName{A Spectral Triple for a Solenoid Based\\ on the Sierpinski Gasket\footnote{This paper is a~contribution to the Special Issue on Noncommutative Manifolds and their Symmetries in honour of~Giovanni Landi. The full collection is available at \href{https://www.emis.de/journals/SIGMA/Landi.html}{https://www.emis.de/journals/SIGMA/Landi.html}}}

\Author{Valeriano AIELLO~$^{\rm a}$, Daniele GUIDO~$^{\rm b}$ and Tommaso ISOLA~$^{\rm b}$}

\AuthorNameForHeading{V.~Aiello, D.~Guido and T.~Isola}

\Address{$^{\rm a)}$~Mathematisches Institut, Universit\"at Bern, Alpeneggstrasse 22, 3012 Bern, Switzerland}
\EmailD{\href{mailto:valerianoaiello@gmail.com}{valerianoaiello@gmail.com}}

\Address{$^{\rm b)}$~Dipartimento di Matematica, Universit\`a di Roma ``Tor Vergata'', I--00133 Roma, Italy}
\EmailD{\href{mailto:guido@mat.uniroma2.it}{guido@mat.uniroma2.it}, \href{mailto:isola@mat.uniroma2.it}{isola@mat.uniroma2.it}}

\ArticleDates{Received June 23, 2020, in final form February 10, 2021; Published online March 02, 2021}

\Abstract{The Sierpinski gasket admits a locally isometric ramified self-covering. A semifinite spectral triple is constructed on the resulting solenoidal space, and its main geometrical features are discussed.}

\Classification{58B34; 28A80; 47D07; 46L05}

\renewcommand{\thefootnote}{\arabic{footnote}}
\setcounter{footnote}{0}

\section{Introduction}
In this note, we introduce a semifinite spectral triple on the $C^*$-algebra of continuous functions on the solenoid associated with a self-covering of the Sierpinski gasket. Such triple is finitely summable, its metric dimension coincides with the Hausdorff dimension of the gasket, and the associated non-commutative integral coincides up to a constant with a Bohr--F{\o}lner mean on the solenoid, hence reproduces the suitably normalized Hausdorff measure on periodic functions. The open infinite Sierpinski fractafold with a unique boundary point considered by Teplyaev~\cite{Tep} embeds continuously as a dense subspace of the solenoid, and the Connes distance restricted to such subspace reproduces the geodesic distance on such fractafold.
On the one hand, this shows that our spectral triple describes aspects of both local and coarse geometry~\cite{RoeLN}.
On the other hand, this implies that the topology induced by the Connes distance, being non compact, does not coincide with the weak$^*$-topology on the states of the solenoid algebra, as we call the $C^*$-algebra of continuous functions on the solenoid. This means that the solenoid, endowed with our spectral triple, is not a quantum metric space in the sense of Rieffel~\cite{Rieffel}.

Related research concerning projective limits of (possibly quantum) spaces and the associated solenoids appeared recently in the literature. In the framework of noncommutative geometry, we~mention:
\cite{LaPa}, where projective families of compact quantum spaces have been studied, showing their convergence to the solenoid w.r.t.\ the Gromov--Hausdorff propinquity distance;
\cite{AGI01}, where, in the same spirit as in this note, a semifinite spectral triple has been associated with the projective limit generated by endomorphisms of $C^*$-algebras associated with commutative and noncommutative spaces;
\cite{DGMW}, where a spectral triple on the stable Ruelle algebra for Wieler solenoids has been considered and its unboundedd KK-theory has been studied, based on the Morita equivalence between the stable Ruelle algebra and a Cuntz--Pimsner algebra. In the same paper these techniques are used for the study of limit sets of regular self-similar groups (cf.~\cite{Nekra}).

When fractals are concerned, we mention the projective family of finite coverings of the octahedron gasket considered in~\cite{Stri2009}, where, as in our present situation, an intermediate infinite fractafold between the tower of coverings and the projective limit is considered. Periodic and almost periodic functions on the infinite fractafold are considered, and a Fourier series description for the periodic functions is given, based on periodic eigenfunctions of the Laplacian (cf.\ also~\cite{RuStri} for higher-dimensional examples). Let us remark that such coverings, as the ones considered in~this paper, are not associated with groups of deck transformations.

The starting point for the construction of this paper is the existence of a locally isometric ramified three-fold self-covering of the Sierpinski gasket with trivial group of deck transformations. Such self-covering gives rise to a projective family of coverings, whose projective limit is by definition a solenoid. Dually, the algebras of continuous functions on the coverings form an~injective family, whose direct limit (in the category of $C^*$-algebras) is the solenoid algebra.
In~\cite{AGI01} we already considered various examples of self-coverings or, dually, of endomorphisms of some $C^*$-algebras, most of which were regular finite self-coverings. There we constructed a~spectral triple on the solenoid algebra as a suitable limit of~spectral triples on the algebras of continuous functions on the coverings. Given a spectral triple on the base space, attaching a~spectral triple to a finite covering is not a difficult task, and in our present case consists simply in ``dilating'' the triple on the base gasket so that the projections are locally isometric. However, there is no commonly accepted procedure to define a limit of spectral triples. Since the method used in~\cite{AGI01} cannot be used here (see below), we follow another route, in a sense spatializing the construction, namely showing that there exists an open fractafold which is intermediate between the projective family of coverings and the solenoid.
 More precisely, such fractafold space turns out to be an~infinite covering of each of the finite coverings of the family, and embeds in~a~continuous way in the solenoid. In this way all the algebras (and their direct limit) will act on a suitable $L^2$-space of the open fractafold, as do the Dirac operators of the associated spectral triples. In~this way the limiting Dirac operator is well defined, but the compact resolvent property will be~lost.

Let us notice here that we are not constructing a spectral triple on the open fractafold, where a weaker compact resolvent property (cf.~\cite[Chapter~IV, Remark~12]{ConnesBook}) is retained, namely $f(D^2+I)^{-1/2}$ is a compact operator, where $D$ is the Dirac operator and $f$ is any function with compact support on the fractafold.
Since we are constructing a spectral triple on the solenoid, which is a compact space, the weaker form does not help.

In order to recover the needed compactness of the resolvent, we use a procedure first proposed by J. Roe for open manifolds with an amenable exhaustion in~\cite{Roe,Roe-2}, where, based on the observation that the von~Neumann trace used by Atiyah~\cite{Atiyah} for his index theorem for covering manifolds can be reformulated in the case of amenable groups via the F{\o}lner condition, he considered amenable exhaustions on open manifolds and constructed a trace for finite-propagation operators acting on sections of a fiber bundle on the manifold via a renormalization procedure.
Unfortunately such trace is not canonical, since it depends on a generalized limit procedure. However, in the case of infinite self-similar CW-complexes, it was observed in~\cite{CGIs01} that such trace becomes canonical when restricted to the $C^*$-algebra of geometric operators.

We adapt these results to our present context, namely we replace the usual trace with a~renor\-ma\-li\-zed trace associated with an exhaustion of the infinite fractafold. Such trace comes together with a noncommutative $C^*$-algebra, the algebra of geometric operators, which is similar in spirit to the Roe $C^*$-algebras of coarse geometry~\cite{HiRo,Roe,Roe-2,Roe96,WiYu}. This algebra contains the solenoid algebra, and the limiting Dirac operator is affiliated to it in a suitable sense. Such Dirac operator turns out to be $\tau$-compact w.r.t.~the renormalized trace. We refer to~\cite{CGIs01, GuIs7} for an analogous construction of the $C^*$-algebra and of a canonical trace based on the self-similarity structure.

As discussed above, the starting point for the construction of a spectral triple on the solenoid algebra is the association of a spectral triple to the fractal known as the Sierpinski gasket~\cite{Sierpinski}.

The study of fractal spaces from a spectral, or noncommutative, point of view has now a~long history, starting from the early papers of Kigami and Lapidus~\cite{KiLa, La94,La97}.
As for the spectral triples, various constructions have been considered in the literature, mainly based on~``small'' triples attached to specific subsets of the fractal, following a general procedure first introduced by Connes, then considered in~\cite{GuIs9,GuIs10}, and subsequently abstracted in~\cite{ChIv07}.
More precisely, the spectral triple on the Cantor set described by Connes~\cite{ConnesBook} inspired two kinds of~spectral triples for various families of fractals in~\cite{GuIs9,GuIs10}. These triples were further analysed in~\cite{GuIs16} for the class of nested fractals. Such specral triples are obtained as direct sums of triples on two points (boundary points of an edge in some cases), and we call them discrete spectral triples. We~then mention some spectral triples obtained as direct sums of spectral triples on 1-dimensional subsets, such as those considered in~\cite{Arau,CIL,CIS,CGIS02,LaSa}, where the 1-dimensional subsets are segments, circles or quasi-circles. Discrete spectral triples give a good description of metric aspects of the fractal, such as Hausdorff dimension and measure and geodesic distance, and, as shown in~\cite{GuIs1}, may also reconstruct the energy functional (Dirichlet form) on the fractal, but are not suited for the study of K-theoretical properties since the pairing with K-theory is trivial. Conversely, spectral triples based on segments or circles describe both metric and K-theoretic properties of~the fractal but can't be used for describing the Dirichlet form. Finally, the spectral triple based on quasi-circles considered in~\cite{CGIS02} describes metric and K-theoretic aspects together with the energy form, but requires a rather technical approach.

In the present paper, we make use of the simple discrete spectral triple on the gasket as described in~\cite{GuIs16}, thus obtaining a semifinite spectral triple on the solenoid algebra which recovers the metric dimension and the Bohr--F{\o}lner mean of the solenoid, and the geodesic distance on the infinite fractafold.
Further analysis on the solenoid is possible, e.g., the construction of a Dirichlet form via noncommutative geometry or the study of K-theoretic properties. As explained above, the latter step will require a different choice of the spectral triple on the base gasket, such as the triples considered in~\cite{CIL,CIS,CGIS02}, which admit a non-trivial pairing with the K-theory of the gasket.

As already mentioned, our aim here is to show that the family of spectral triples on the finite coverings produces a spectral triple on the solenoidal space. In the examples considered in~\cite{AGI01}, the family of spectral triples had a simple tensor product structure, namely the Hilbert spaces were a tensor product of the Hilbert space $\ch$ for the base space and a finite dimensional Hilbert space, and the Dirac operators could be described as (a finite sum of) tensor product operators. Then the ambient $C^*$-algebra turned out to be a product of $\cb(\ch)$ and a UHF algebra, allowing a GNS representation w.r.t.~a~semifinite trace.

In the example treated here we choose a different approach since two problems forbid such simple description. The first is a local problem, due to the ramification points. This implies that the algebra of a covering is not a free module on the algebra of the base space; in particular,
functions on a covering space form a proper sub-algebra of the direct sum of finitely many copies of the algebra for the base space.
The second is a non-local problem which concerns the Hilbert spaces, which are $\ell^2$-spaces on edges, and the associated operator algebras. Indeed, the Hilbert spaces of the coverings cannot be described as finite sums of copies of the Hilbert space on the base space due to the appearance of longer and longer edges on larger and larger coverings.

We conclude this introduction by mentioning two further developments of the present analysis.

First, the construction of the spectral triple on the solenoid algebra allows the possibility of~lifting a spectral triple from a $C^*$-algebra to the crossed product of the C*-algebra with a~single endomorphism~\cite{AGI02}, thus generalising the results on crossed products
with an automorphism group considered in~\cite{BMR,Skalski,Paterson}.

Second, we observe that the construction given in the present paper goes in the direction of~possibly defining a $C^*$-spectral triple, in which the semifinite von Neumann algebra is replaced by a $C^*$-algebra with a trace to which both the Dirac operator and the ``functions'' on the non-commutative space are affiliated, where the compactness of the resolvent of the Dirac operator is measured by the trace on the $C^*$-algebra, cf.~also~\cite{GuIs4}.

This paper is divided in six sections. After this introduction, Section~\ref{sec2} contains some pre\-li\-mi\-nary notions on fractals and spectral triples, Section~\ref{sec3} describes the geometry of the ramified covering and the corresponding inductive structure, together with its functional counterpart given by a family of compatible spectral triples. Section~\ref{sec4} concerns the self-similarity structure of the Sierpinski solenoid, whence the description of the inductive family of $C^*$-algebras as algebras of bounded functions on the fractafold. The Section~\ref{sec5} describes the algebra of geometric operators and the construction of a semicontinuous semifinite trace on it. Finally, the semifinite spectral triple together with its main features are contained in Section~\ref{sec6}.

\section{Preliminaries}\label{sec2}

In this section we shall briefly recall various notions that will be used in the paper.
Though these notions are well known among the experts, our note concerns different themes, namely spectral triples in noncommutative geometry and nested fractals (the Sierpinski gasket in particular), so~that we decided to write this section with the aim of helping readers with different background to follow the various arguments, by collecting here the main notions and results that will be useful in the following.

\subsection{Spectral triples}

The notion of spectral triple plays a key role in Alain Connes'non\-com\-mu\-ta\-ti\-ve geometry~\cite{ConnesBook,GBVF}. Basically, it consists of a triple $(\cl,\ch,D)$, where $\cl$ is a *-algebra acting faithfully on the Hilbert space $\ch$, and $D$ is an unbounded self-adjoint operator on $\ch$ satisfying the properties
\begin{itemize}\itemsep=0pt
\item[$(1)$] $\big(1+D^2\big)^{-1/2}$ is a compact operator,
\item[$(2)$] $\pi(a)\mathcal{D} (D) \subset \mathcal{D} (D)$, and $[D,\pi(a)]$ is bounded for all $a\in \cl$.
\end{itemize}
We shall also say that $(\cl,\ch,D)$ is a spectral triple on the $C^*$-algebra $\ca$ generated by~$\cl$.

Such triple is meant as a generalization of a compact smooth manifold, the algebra $\cl$ replacing the algebra of smooth functions, the Hilbert space describing a vector bundle (a spin bundle indeed) on which the algebra of functions acts, and the operator~$D$ generalizing the notion of Dirac operator.
Further structure may be added to the properties above, allowing deeper analysis of the geometric features of the noncommutative manifold, but these are not needed in~this paper.

Property $(2)$ above allows the definition of a (possibily infinite) distance (Connes distance) on the state space of the $C^*$-algebra $\ca$ generated by $\cl$ , defined as
\begin{gather*}
d(\f,\psi)=\sup\{|\f(a)-\psi(a)|\colon \|[D,a]\|\leq1,\, a\in \cl\}.
\end{gather*}
When the Connes distance induces the weak$^*$-topology on the state space, the seminorm $\|[D,a]\|$ on~$\ca$ is called a Lip-norm (cf.~\cite{Rieffel}) and the algebra $\ca$ endowed with the Connes distance is a~quantum metric space.

A spectral triple is called finitely summable if $\big(1+D^2\big)^{-s}$ has finite trace for some $s>0$, in this case the abscissa of convergence~$d$ of the function $\tr\big(1+D^2\big)^{-s}$ is called the metric dimension of~the triple. Then the logarithmic singular trace introduced by Dixmier~\cite{Dix} may be used to define a noncommmutative integral on~$\ca$.
Let us denote by $\{\m_n(T)\}$ the sequence (with multiplicity) of~singular values of the compact operator~$T$, arranged in decreasing order.
Then, on the positive compact operators for which the sequence $\sum_{k=1}^n\m_n(T)$, is at most logarithimically divergent, we may consider the positive functional
\begin{gather*}
\tr_\om(T)=\Lim_\om\frac{\sum_{k=1}^n\m_n(T)}{\log n},
\end{gather*}
where $\Lim_\om$ is a suitable generalized limit. Such functional extends to a positive trace on $\cb(\ch)$ which vanishes on trace class operators, and is called Dixmier (logarithmic) trace.

If $\big(1+D^2\big)^{-d}$ is in the domain of the Dixmier trace, one defines the following noncommutative integral:
\begin{gather*}%\label{NCint}
\oint a=\tr_\om\big (a\big(I+D^2\big)^{-d/2}\big),\qquad a\in\ca.
\end{gather*}
When the function $\big(1+D^2\big)^{-s}$ has a finite residue for $s=d$, such residue turns out to coincide, up to a constant, with the Dixmier trace, which therefore does not depend on the generalized limit procedure (cf.~\cite{ConnesBook}, and~\cite[Theorem~3.8]{CPS}):
\begin{gather*}
d\cdot\tr_\om \big(a\big(I+D^2\big)^{-d/2}\big)={\Res}_{s=d} \tr(a|D|^{-s}).
\end{gather*}
We note in passing that spectral triples may also describe non-compact smooth manifolds, with the algebra~$\cl$ describing smooth functions with compact support and property~(1) replaced by~$a\big(1+D^2\big)^{-1/2}$ is a compact operator for any $a\in\cl$.

\subsection{Semifinite spectral triples}\label{SemST}
The notion of spectral triple has been generalized to the semifinite case, by replacing the ambient algebra $\cb(\ch)$ with a semifinite von~Neumann algebra $\cam$ endowed with a normal semifinite faithful trace~$\t$.
We recall that an operator $T$ affiliated with $(\cam,\tau)$ is called $\tau$-compact if its generalized s-number function $\mu_t(T)$ is infinitesimal or, equivalently, if $\tau(e_{(t,\infty)}(T))<\infty$, for~any $t>0$ (cf.~\cite[Section~1.8, p.~34]{Fack}, \cite[Proposition~3.2]{FK}).

\begin{dfn} [\cite{CaPhi1}]\label{def:SFtriple}
	An odd semifinite spectral triple $(\cl,\cam,D)$ on a unital C$^*$-algebra $\ca$ is given by a unital, norm-dense, $^*$-subalgebra $\cl\subset\ca$, a semifinite von Neumann algebra $(\cam,\tau)$, acting on a (separable) Hilbert space $\ch$, a faithful representation $\pi\colon\ca\to\cb(\ch)$ such that $\pi(\ca)\subset\cam$, and an unbounded self-adjoint operator $D \widehat{\in} \cam$ such that
	\begin{itemize}\itemsep=0pt
\item[$(1)$] $\big(1+D^2\big)^{-1/2}$ is a $\tau$-compact operator,
\item[$(2)$] $\pi(a)\mathcal{D} (D) \subset \mathcal{D} (D)$, and $[D,\pi(a)] \in\cam$, for all $a\in\cl$.
\end{itemize}
\end{dfn}
As in the type $I$ case, such triple is called finitely summable if $\big(1+D^2\big)^{-s}$ has finite trace for some $s>0$, and $d$ denotes the abscissa of convergence of the function $\tau\big(1+D^2\big)^{-s}$, and is called the metric dimension of the triple.
The logarithmic Dixmier trace associated with the normal trace $\t$ may be defined in this case too, (cf.~\cite{CPS, GuIs1}) and, when the function $\big(1+D^2\big)^{-s}$ has a~finite residue for $s=d$, the equality $d\cdot\tr_\omega\big(a|D|^{-d}\big)={\Res}_{s=d} \tr\big(a|D|^{-s}\big)$ still holds \cite[Theorem~3.8]{CPS}.

\subsection{Self-similar fractals}
Let $\O:= \{w_i \colon i=1,\ldots,k \}$ be a family of contracting similarities of $\br^{N}$, with scaling para\-me\-ters~$\{\l_i\}$. The unique non-empty compact subset $K$ of $\br^{N}$ such that $K = \bigcup_{i=1}^{k} w_i(K)$ is called the {\it self-similar fractal} defined by $\{w_i \}_{i=1,\ldots,k}$. For any $i\in\{1,\ldots,k\}$, let $p_i\in\br^N$ be the unique fixed-point of $w_i$, and say that $p_i$ is an essential fixed-point of $\O$ if there are $i',j,j'\in\{1,\ldots,k\}$ such that $i'\neq i$, and $w_j(p_i)=w_{j'}(p_{i'})$. Denote by $V_0(K)$ the set of essential fixed-points of $\O$, and let $E_0(K):=\{ (p,q)\colon p,q\in V_0,\ p\neq q\}$. Observe that $(V_0,E_0)$ is a directed finite graph whose edges are in $1:1$ correspondence with ordered pairs of distinct vertices.
\begin{dfn}
We call an element of the family $\{w_{i_1}\cdots w_{i_k}(K)\colon k\geq0\}$ a {\it cell}, and call its diameter the size of the cell.
We call an element of the family $E(K)=\{w_{i_1}\cdots w_{i_k}(e)\colon k\geq0$, $e\in E_{0}(K)\}$ an {\it $($oriented$)$ edge} of $K$. We denote by $e^-$ resp.~$e^+$ the source, resp.~the target of~the oriented edge $e$.
\end{dfn}
As an example, the Sierpinski gasket is the self-similar fractal determined by 3 similarities with scaling parameter 1/2 centered in the vertices of an equilateral triangle (see Fig.~\ref{fig:Gasket}).

\begin{figure}[t]\centering
\includegraphics[scale=1]{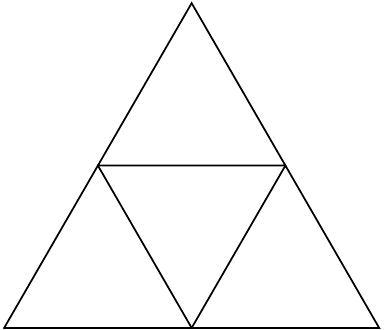}
\qquad\qquad
\includegraphics[scale=1]{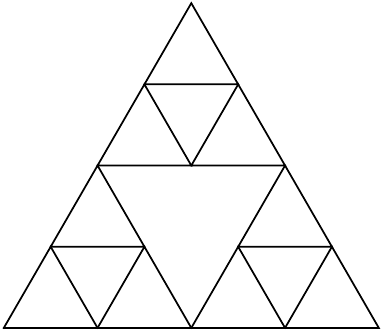}
\\[2ex]
\includegraphics[scale=1]{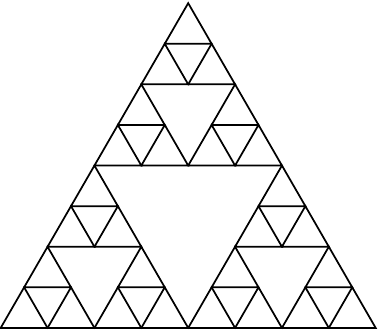}
\qquad\qquad
\includegraphics[scale=1]{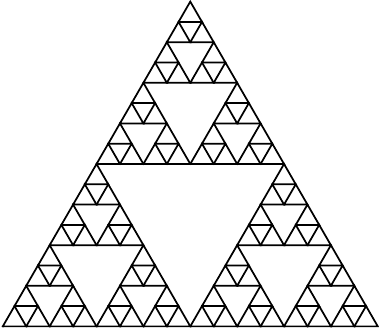}
\caption{The first four steps of the construction of the gasket.}\label{fig:Gasket}
\end{figure}

Under suitable conditions, the Hausdorff dimension $d_H$ of a self-similar fractal coincides with its scaling dimension, namely with the only positive number $d$ such that $\sum_{i=1}^k\l_i^d=1$, therefore when all scaling parameters coincide with $\l$ we have $d_H=\frac{\log k}{\log(1/\l)}$. In particular, the Hausdorff dimension of the Sierpinski gasket is $\frac{\log 3}{\log2}$.
We note in passing that one of the most important aspects of the Sierpinski gasket and of more general classes of fractals is the existence of a self-similar diffusion, associated with a Dirichlet form, see, e.g.,~\cite{Kiga2}.
Even though Dirichlet forms on~fractals can be recovered in the noncommutative geometry framework~\cite{GuIs16}, and in particular by~means of the spectral triples which we use in this paper, we do not analyse this aspect in the present note.

In~\cite{GuIs10} discrete spectral triples have been introduced on some classes of fractals, generalizing an example of Connes in~\cite[Chapter~4.3, Example~23]{ConnesBook}. Such triples have been further studied in~\cite{GuIs16} for nested fractals.
On a self-similar fractal $K$, the triple $ (\cl,\ch,D)$ on the $C^*$-algebra $\ca=\cc(K)$ is defined as follows:
\begin{dfn}\label{STnested}\qquad
\begin{itemize}\itemsep=0pt
\item[$(a)$] $\ch=\ell^2(E(K))$,
\item[$(b)$] $\ca$ acts on the Hilbert space as $\rho(f)e=f(e^+)e$, $f\in\ca_n$, $e\in E_n$,
\item[$(c)$] $F$ is the orientation-reversing map on edges,
\item[$(d)$] $D$ maps an edge $e\in E(K)$ to $\length(e)^{-1}Fe$,
\item[$(e)$] $\cl$ is given by the elements $ f\in \ca$ such that $\|[D,\rho(f)]\|<\infty$.
\end{itemize}
\end{dfn}
It turns out that $\cl$ coincides with the algebra of Lipschitz functions on $K$, hence is dense in~$\ca$, and the seminorm $L(f):=\|[D,f]\|$ is a Lip-norm. By Theorem 3.3 in~\cite{GuIs16}, see also Remark~2.11 in~\cite{GuIs10}, the triple $(\cl,\ch,D)$ is a finitely summable spectral triple on $\ca$, its metric dimension coincides with the Hausdorff dimension, and the noncommutative integral recovers the Hausdorff measure up to a constant:
\begin{gather}\label{fractalNCint}
 \oint f=\tr_\omega\big(f|D|^{-d}\big) = \frac{1}{ \log k} \sum_{e\in E_0(K)} \ell(e)^d \int_K f\, {\rm d}H_d, \qquad f\in C(K),
\end{gather}
where $H_d$ denotes the normalized Hausdorff measure on the fractal $K$.
Moreover, in some cases, and in particular for the Sierpinski gasket, the Connes distance induced by the Lip-norm $L(f):=\|[D,f]\|$ coincides with the geodesic distance on the points of the gasket $K$, see~\cite[Corollary~5.14]{GuIs16}.

\subsection{Covering fractafolds and solenoids}

Generally speaking, a solenoid is the inverse limit of a projective family of coverings of a given space~\cite{McCord}. Dually, the solenoid algebra is the direct limit of the family of algebras of continuous functions on the spaces of the projective family. In this sense the notion of solenoid makes sense for injective families of $C^*$-algebras, cf., e.g.,~\cite{AGI01} for sequences generated by a single endomorphism and~\cite{LaPa} for sequences of compact quantum spaces. Other examples of the treatment of solenoids in the recent literature have been mentioned in the introduction.

The notion of fractafold as a connected Hausdorff topological space such that every point has a neighborhood homeomorphic to a neighborhood in a given fractal has been introduced in~\cite{Stri2003}, even though examples of such notion were already considered before, e.g., in~\cite{BaPe,Stri1996,Tep}.
In some cases projective families of covering fractafold spaces related to the Sierpinski gasket have been considered.

Since the gasket does not admit a simply connected covering, one may consider coverings where more and more cycles are unfolded, in particular consider the regular infinite abelian covering $S_n$ where all the cycles of size at least $2^{-n}$ are unfolded. Each of those is a closed fractafold (with boundary) and they form a projective family.
The associated solenoid~$S_\infty$, i.e., the projective limit, which turns out to be an abelian counterpart of the Uniform Universal Cover introduced by Berestovskii and Plaut~\cite{BerPla}, has been considered in~\cite{CGIS13}, where it is shown that any locally exact 1-form on the gasket possesses a potential on~$S_\infty$.

Another projective family of covering fractafolds has been considered in~\cite{Stri2009}, each element of~the family being a compact finite covering of the octahedral fractafold modeled on the gasket. Any element of the family is covered by the infinite Sierpinski gasket with a unique boundary point, which we call $K_\infty$ here (see Fig.~\ref{fig:infiniteBlowup}), considered in~\cite[Lemma~5.11]{Tep}. The solenoid associated with the projective family is also mentioned explicitly in~\cite{Stri2009}, together with the dense embedding of $K_\infty$ in it, and also a Bohr--F{\o}lner mean on the solenoid is considered (p.~1199).
\begin{figure}[t]\centering
\includegraphics[scale=1.03]{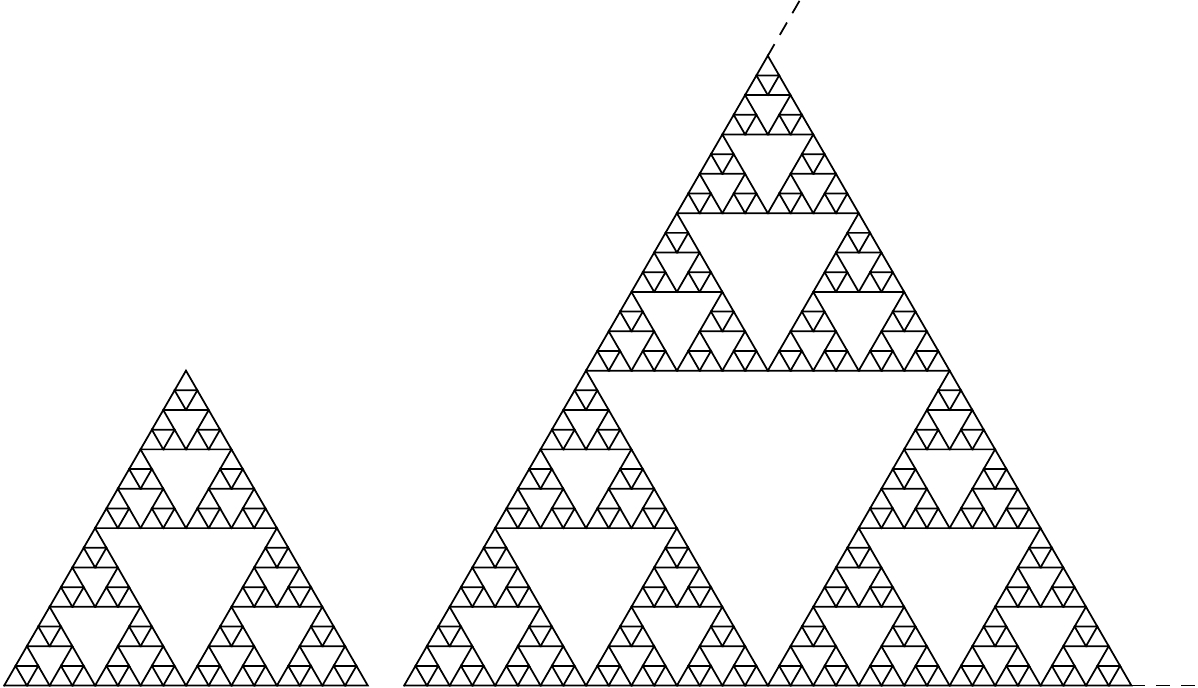}
\caption{The gasket and its infinite blowup.}\label{fig:infiniteBlowup}
 \end{figure}

In the present paper a self-covering of the gasket gives rise to a projective family of finite ramified coverings, the fractafold $K_\infty$ projects onto each element of the family and embeds densely in the solenoid, and we recover the Bohr--F{\o}lner mean on the solenoid via a noncommutative integral.

\section{A ramified covering of the Sierpinski gasket}\label{sec3}

Let us choose an equilateral triangle of side 1 in the Euclidean plane with vertices $v_0$, $v_1$, $v_2$ (numbered in a counterclockwise order) and consider the associated Sierpinski gasket as in the previous section, namely the set $K$ such that
\begin{gather*}
K=\bigcup_{j=0,1,2}w_j(K),
\end{gather*}
where $w_j$ is the dilation around $v_j$ with contraction parameter $1/2$.
Clearly, for the cell $C=w_{i_1}\cdots w_{i_k}(K)$, $\size(C)=2^{-k}$ and, if $e_0\in E_{0}(K)$ and $e=w_{i_1}\cdots w_{i_k}(e_0)$, $\length(e)=2^{-k}$.

\looseness=1
5In the following we shall set $K_0:=K$, $E_0=E_0(K)$, $K_n=w_0^{-n}K_0$.
Let us now consider the middle point $x_{i,i+1}$ of the segment $\big(w_0^{-1}v_i,w_0^{-1}v_{i+1}\big)$, $i=0,1,2$, the map $R_{i+1,i}\colon w_0^{-1}w_{i}K\to w_0^{-1}w_{i+1}K$ consisting of the rotation of $\frac43\pi$ around the point $x_{i,i+1}$, $i=0,1,2$, and observe that
\begin{gather}\label{IdOnCells}
R_{i,i+2}\circ R_{i+2,i+1}\circ R_{i+1,i} = {\rm id}_{w_0^{-1}w_i K},\qquad i=0,1,2.
\end{gather}
Setting $R_{i,i+1} = R_{i+1,i}^{-1}$, the previous identities may also be written as
\begin{gather*}
R_{i+2,i+1}\circ R_{i+1,i} = R_{i+2,i},\qquad i=0,1,2.
\end{gather*}

We then construct the map $p\colon K_1\to K$ given by
\begin{gather*}
p(x)=
\begin{cases}
x,&x\in K,
\\
R_{0,1}(x),&x\in w_0^{-1}w_1 K,
\\
R_{0,2}(x),&x\in w_0^{-1}w_2 K,
\end{cases}
\end{gather*}
and observe that this map, which appears to be doubly defined in the points $x_{i,i+1}$, $i=0,1,2$, is indeed well defined (see Fig.~\ref{fig:covering}).

\begin{figure}[t]\centering%
\scalebox{3.8}{\def\trianglewidth{2cm}%
\pgfdeclarelindenmayersystem{Sierpinski triangle}{%
\symbol{X}{\pgflsystemdrawforward}%
\symbol{Y}{\pgflsystemdrawforward}%
\rule{X -> X-Y+X+Y-X}%
\rule{Y -> YY}%
}%
\foreach \level in {7}{%
\tikzset{%
l-system={step=\trianglewidth/(2^\level), order=\level, angle=-120}%
}%
\begin{tikzpicture}%
\fill [black] (0,0) -- ++(0:\trianglewidth) -- ++(120:\trianglewidth) -- cycle;%
\draw [draw=none] (0,0) l-system [l-system={Sierpinski triangle, axiom=X},fill=white];%
\draw[line width=0.05mm, <-] (.75,-.1) to[out=-90,in=-90] (1.25,-.1);%
\draw[line width=0.05mm, <-] (0.25,.6) to[out=135,in=120] (.5,1);%
\node (bbb) at (-.05,-0.075) {$\scalebox{.25}{$v_0$}$};%
\node (bbb) at (.395,.87) {$\scalebox{.25}{$x_{2,0}$}$};%
\node (bbb) at (1.6,.87) {$\scalebox{.25}{$x_{1,2}$}$};%
\node (bbb) at (1,1.8) {$\scalebox{.25}{$w_0^{-1}v_2$}$};%
\node (bbb) at (1,-.075) {$\scalebox{.25}{$x_{0,1}=v_1$}$};%
\node (bbb) at (2.05,-.075) {$\scalebox{.25}{$w_0^{-1}v_1$}$};%
\node (bbb) at (0.1,.87) {$\scalebox{.25}{$R_{0,2}$}$};%
\node (bbb) at (1,-.35) {$\scalebox{.25}{$R_{0,1}$}$};%
\end{tikzpicture}%
}%
}%
\vspace{-4ex}\caption{The covering map $p\colon K_1\to K$.}\label{fig:covering}
\end{figure}

The following result is easily verified.
\begin{prop}
The map $p$ is a well defined continuous map which is a ramified covering, with ramification points given by $\{x_{i,i+1},i=0,1,2\}$. Moreover, the covering map is isometric on suitable neighbourhoods of the non-ramification points.
\end{prop}
Since $K_1$ and $K$ are homeomorphic, this map may be seen as a self-covering of the gasket.
The map $p$ gives rise to an embedding $\a_{1,0}\colon \cc(K)\to\cc(K_1)$, hence, following~\cite{Cuntz}, to an inductive family of C$^*$-algebras $\ca_n=\cc(K_n)$, whose inductive limit $\ca_\infty$ consists of continuous function on the solenoidal space based on the gasket.
As in Definition~\ref{STnested}, we consider the triple $ (\cl_n,\ch_n,D_n)$ on the $C^*$-algebra $\ca_n$, $n\geq0$, where
$\ch_n=\ell^2(E_n)$, $E_n=\{w_0^{-n}e,\, e\in E_0\}$ (the set of oriented edges in $K_n$).

Let us also note that, since the covering projections are locally isometric and any Lip-norm $L_m(f)=\|[D_m,f]\|$ associated with the triple $(\ca_m,\ch_m,D_m)$ produces the geodesic distance on~$K_m$, we get $L_{m+q}(\a_{m+q,m}(f))=L_m(f)$, namely we obtain a seminorm on the algebraic inductive limit of the $\ca_n$'s.

\section[A groupoid of local isometries on the infinite Sierpinski fractafold]
{A groupoid of local isometries on the infinite \\Sierpinski fractafold}\label{sec4}

Let us consider the infinite fractafold $K_\infty=\cup_{n\geq0}K_n$~\cite{Tep} endowed with the Hausdorff measure $\vol$ of dimension $d=\frac{\log3}{\log2}$ normalized to be 1 on $K=K_0$, with the exhaustion $\{K_n\}_{n\geq0}$, and with the family of local isometries $R=\big\{R^n_{i+1,i},R^n_{i,i+1}\colon i=0,1,2, n\geq 0\big\}$, where $R^n_{i,j} = w_0^{-n}R_{i,j}w_0^{n}\colon C^n_j \to C^n_i$, and $C^n_i := w_0^{-n-1}w_iK$, $n\geq 0$, $i,j\in\{0,1,2\}$. We also denote by $s(\g)$ and $r(\g)$ the domain and range of the local isometry~$\g$. Such local isometries act on points and on oriented edges of $K_\infty$.

We say that the product of the two local isometries $\g_1$, $\g_2\in R$ is defined if $\g_2^{-1}(s(\g_1))\cap$ $s(\g_2) \neq \varnothing$. In this case we consider the product
\begin{gather*}
\g_1\cdot\g_2\colon\ \g_2^{-1}(s(\g_1))\cap s(\g_2)\to r\big(\g_1|_{s(\g_1))\cap r(\g_2)}\big).
\end{gather*}

We then consider the family $\cg$ consisting of all (the well-defined) finite products of isometries in $R$. Clearly, any $\g$ in $\cg$ is a local isometry, and its domain and range are cells of the same size. We set $\cg_n=\{g\in\cg\colon s(\g)\,\&\,r(\g)$ are cells of size $2^n\}$, $n\geq0$.

\begin{prop}%\label{E!}
For any $n\geq0$, $C_1$, $C_2$ cells of size $2^n$, $\exists!\, \g\in\cg_n$ such that $s(\g)=C_1$, $r(\g)=C_2$. In particular, if $C$ has size $2^n$, the identity map of $C$ belongs to $\cg_n$, $n\geq0$.
\end{prop}
\begin{proof}
It is enough to show that for any cell $C$ of size $2^n$ there exists a unique $\g\in\cg_n$ such that $\g\colon C\to K_n$.
For any cell $C$, let $m=\level(C)$ be the minimum number such that $C\subset K_m$.
We~prove the existence: if $C$ has size $2^n$ and $\level(C)=m>n$, then $C\subset C^{m-1}_i$, for some $i=1,2$, hence $R^{m-1}_{0,i}(C) \subset K_{m-1}$. Iterating, the result follows.
The second statement follows directly by equation~\eqref{IdOnCells}.

As for the uniqueness, $\forall\, n\geq0$, we call $R^n_{i,0}$ ascending, $i=1,2$, $R^n_{0,i}$ descending, $i=1,2$, $R^n_{i,j}$~constant-level, $i,j\in\{1,2\}$.
Indeed, if $C \subset s(R^n_{i,0})$, then $\level(C)\leq n$ and $\level(R^n_{i,0}(C)) = n+1$; if $C \subset s(R^n_{0,i})$, then $\level(C)=n+1$, $\level(R^n_{0,i}(C))\leq n$ and $\level(R^n_{j,i}(C))= n+1$, $i,j\in\{1,2\}$, $n\geq0$.

The following facts hold:
\begin{itemize}\itemsep=0pt
\item The product $R^n_{l,k} \cdot R^m_{j,i}$ of two constant-level elements $R^n_{l,k}$, $R^m_{j,i}$ is defined iff $n=m$ and~$k=j$, therefore any product of constant-level elements in $R$ is either the identity map on~the domain or coincides with a single constant-level element.
\item Any product of constant level elements in $R$ followed by a descending element coincides with a single descending element: indeed, if the product of constant level elements is the identity, the statement is trivially true; if it coincides with a single element, say $R^n_{i,j}$ with~$i,j\in\{1,2\}$, then, by compatibility, the descending element should be $R^n_{0,i}$ so that the product is $R^n_{0,i},$ by equation~\eqref{IdOnCells}.
\item Given a cell $C$ with $\size (C)=2^n$ and $\level(C)>n$, the exists a unique descending element $\g \in R$ such that $C \subset s(\g)$: indeed, if $m=\level(C)$, then $C\subset C^{m-1}_i$, for some $i\in\{1,2\}$. The only descending element is then $\g=R^{m-1}_{0,i}$.
\item Any product of an ascending element followed by a descending one is the identity on the domain: indeed if the ascending element is $R^n_{i,0}$, then, by compatibility, the descending element should be $R^n_{0,i}$.
\end{itemize}

Now let $\size(C)=2^n$, $\g \in\cg_n$ such that $\g\colon C\to K_n$, $\g=\g_p\cdot\g_{p-1}\cdots\g_2\cdot\g_1$, where $\g_j\in R$, $1\leq j\leq p$. Since $\level(C) \geq \level(K_n) = n$, for any possible ascending element $\g_i$ there should be a $j>i$ such that $\g_j$ is descending. If $i+q$ is the minimum among such $j$'s, all terms $\g_j$, $i<j<i+q$, are constant-level, hence the product $\g_{i+q}\cdot\g_{i+q-1}\cdots \g_{i}={\rm id}_{s(\g_i)}$. Then, we note that $\g_p$ can only be descending. As a consequence, $\g$ can be reduced to a product of descending elements, and, by the uniqueness of the descending element acting on a given cell, we get the result.
\end{proof}

Let us observe that each $\cg_n$, and so also $\cg$, is a groupoid under the usual composition rule, namely two local isometries are composable if the domain of the first coincides with the range of the latter.

We now consider the action on points of the local isometries in $\cg$.

\begin{prop}
Let us define $\widetilde\ca_n$ as the algebra
\begin{gather*}
\widetilde\ca_n=\{f\in\cc_b(K_\infty)\colon f( \g(x))= f(x),\, x\in s(\g),\, \g\in\cg_n\}.
\end{gather*}
Then, for any $n\geq0$, the following diagram commutes,
\begin{gather*}
\begin{matrix}
\widetilde\ca_n & \subset & \widetilde\ca_{n+1} \\
\Big\downarrow\iota_n & & \Big\downarrow\iota_{n+1} \\
\ca_n & \mathop{\longrightarrow}\limits^{\a_{n+1,n}} & \ca_{n+1},
\end{matrix}
\end{gather*}
where $\iota_n\colon f\in\widetilde\ca_n \to f|_{K_n} \in \ca_n$ are isomorphisms.
Hence the inductive limit $\ca_\infty$ is isomorphic to a $C^*$-subalgebra of $\cc_b(K_\infty)$.
\end{prop}
\begin{proof}
The request in the definition of $\widetilde\ca_n$ means that the value of $f$ in any point of $K_\infty$ is~determined by the value on $K_n$, while such request gives no restrictions on the values of $f$ on~$K_n$. The other assertions easily follow.
\end{proof}

As shown above, we may identify the algebra $\ca_n$, $0\leq n\leq \infty$, with its isomorphic copy~$\widetilde\ca_n$ in~$\cc_b(K_\infty)$, so that the embeddings $\a_{k,j}$ become inclusions.
Moreover, we may consider
the operator $\widetilde D_n$ on $\ell^2(E_\infty)$, with $E_\infty=\cup_{n\geq0}E_n$, given by $\widetilde D_n e=\length(e)^{-1}Fe$, if $\length(e)\leq2^n$, and $\widetilde D_n e=0$, if $\length(e) >2^n$, where $F$ is defined as in Definition~\ref{STnested}$(c)$.
Then the spectral triples $(\ca_n,\ch_n,D_n)$ are isomorphic to the spectral triples $(\widetilde\ca_n,\ch_n,\widetilde D_n)$, where $\cc_b(K_\infty)$ acts on~the space $\ell^2(E_\infty)$ through the representation $\rho$ given by $\rho(f)e=f(e^+)e$.
\begin{rem}
Because of the isomorphism above, from now on we shall remove the tildes and denote by $\ca_n$ the subalgebras of $\cc_b(K_\infty)$ and by $D_n$ the operators acting on $\ell^2(E_\infty)$.
\end{rem}

\section[The C*-algebra of geometric operators and a tracial weight on it]
{The $\boldsymbol{C^*}$-algebra of geometric operators\\ and a tracial weight on it}\label{sec5}

We now come to the action of local isometries on edges. We shall use the following notation, where in the table below to any subset of edges listed on the left we indicate on the right the projection on the closed subspace spanned by the same subset:
 \begin{table}[h!]\centering
 \caption{Edges and projections.}\vspace{1ex}
 \label{tab:table1}
{\renewcommand{\arraystretch}{1.4}%
\begin{tabular}{l|l}
 \hline
 \multicolumn{1}{c|}{Subsets of $E_\infty$} & \multicolumn{1}{c}{Projections} \\[.5ex]
 \hline%\\[-2ex]
$E_n=\{e\subset K_n\}$, $n\geq0$ &\qquad $P_n$
\\%[.5ex]
$E^{k,p}_n=\big\{e\in E_n\colon 2^k\leq\length(e)\leq2^p\big\}$, for $k\leq p\leq n$ &\qquad $P_n^{k,p}$
\\%[.5ex]
$E^k_n=E^{k,k}_n=\big\{e\in E_n\colon \length(e)=2^k\big\}$, for $k\leq n$ &\qquad $P^k_n$
\\%[.5ex]
$E^{k,p}=\cup_n E^{k,p}_n=\big\{e\in E_\infty\colon 2^k\leq\length(e)\leq2^p\big\}$ &\qquad $P^{k,p}$
\\%[.5ex]
$E^k=E^{k,k}=\big\{e\in E_\infty\colon \length(e)=2^k\big\}$ &\qquad $P^{k}$
\\%[.5ex]
$E_C = \{e\in E_\infty\colon e\subset C\}$, $C$ being a cell &\qquad $P_C$
 \end{tabular}}
\end{table}

Let us note that any local isometry $\g\in\cg$, $\g\colon s(\g) \to r(\g)$, gives rise to a partial isometry~$V_\g$ defined as
\begin{gather*}
V_\g e=\begin{cases}
\g(e), & e\subset s(\g),\\
0, & \mathrm{elsewhere}.
\end{cases}
\end{gather*}
In particular, if $C$ is a cell, and $\g={\rm id}_C$, $V_\g=P_C$.
We then consider the subalgebras $\cb_n$ of~$B(\ell^2(E_\infty))$,
\begin{gather*}
\cb_n=\{V_\g\colon\g\in\cg_m,\,m\geq n\}',\qquad \Balg = \bigcup_n\cb_n,\qquad \cb_\infty=\overline{\Balg},
\end{gather*}
{\sloppy and note that the elements of $\cb_n$ commute with the projections $P_C$, for all cells $C$ s.t.\ \mbox{$\size(C)\geq 2^n$}.
By definition, the sequence $\cb_n$ is increasing, therefore, since the $\cb_n$'s are von Neumann algebras, $\cb_\infty$ is a $C^*$-algebra. Let us observe that, $\forall\, n\geq0$, $\rho(\ca_n)\subset\cb_n$.

}

\begin{dfn}
The elements of the $C^*$-algebra $\cb_\infty$ are called geometric operators.
\end{dfn}

Now consider the hereditary positive cone
\begin{gather*}
\cai_0^+ = \big\{T\in \Balg^+\colon\exists\, c_T\in\br \mathrm{\ such\ that\ }\tr (P_{m}T) \leq c_T \vol(K_m), \, \forall\, m \geq 0 \big\}.
\end{gather*}

\begin{lem}%\label{increasing}
For any $T\in\cai_0^+$, the sequence
$
\frac{\tr (P_{m} T) }{\vol(K_m)}$ is eventually increasing, hence convergent. In particular
\begin{gather}\label{evConst}
\tr(P_{p}^{p} T)=0\qquad \forall\, p>m\Rightarrow \t_0(T)=\frac{\tr (P_{m} T) }{\vol(K_{m})}.
\end{gather}
\end{lem}
\begin{proof}
Let $T\in\cb_n^+$. Then we have, for $m\geq n$,
\begin{gather*}
\tr (P_{{m+1}} T) = \!\!\sum_{e\subset K_{m+1}} (e,Te) = \!\!\sum_{i=0,1,2} \sum_{e\in C^m_i }(e,Te) + \!\sum_{e\in E_{m+1}^{m+1}}\! (e,Te)
= 3\tr (P_{{m}} T) + \tr(P_{m+1}^{m+1} T),
\end{gather*}
hence
\begin{gather*}%\label{a_m-formula}
\frac{\tr (P_{m+1} T) }{\vol(K_{m+1})}=\frac{\tr (P_{m} T) }{\vol(K_m)}+\frac{\tr(P_{m+1}^{m+1} T)}{\vol(K_{m+1})},
\end{gather*}
from which the thesis follows.
\end{proof}

We then define the weight $\t_0$ on $\cb_\infty^+$ as follows:
\begin{gather*}
\t_0(T) =
\begin{cases}
\displaystyle{\lim_{m\to\infty}\frac{\tr (P_{m} T)}{\vol(K_m)}},& T\in{\cai_0^+},\\
0,& \mathrm{elsewhere}.
\end{cases}
\end{gather*}

The next step is to regularize the weight $\t_0$ in order to obtain a semicontinuous semifinite tracial weight $\t$ on $\cb_\infty$.

\begin{lem} \label{hereditary}
For any $T\in\cai_0^+$, $A\in\Balg$, it holds $ATA^* \in \cai_0^+$, and $\t_0(ATA^*) \leq \|A\|^2 \t_0(T)$.
\end{lem}
\begin{proof}
Let $A\in\cb_n$. Then, for any $m>n$, we have
\begin{gather*}
\tr(P_mATA^*) = \tr(A^*AP_mT) \leq \| A^*A \| \tr(P_mT) \leq \|A\|^2 c_T \vol(K_m),
\end{gather*}
and the thesis follows.
\end{proof}

\begin{prop} \label{def.Qfi}
For all $p\in\bn$, recall that $P^{-p,\infty}$ is the orthogonal projection onto the closed vector space generated by $\big\{ e\in\ell^2(E_\infty)\colon \length(e) \geq 2^{-p} \big\}$, and let $\f_p(T) := \t_0(P^{-p,\infty}TP^{-p,\infty})$, $\forall\, T\in\cb_\infty^+$. Then $P^{-p,\infty}\in\cb_0$, $\f_p$ is a positive linear functional, and $\f_p(T) \leq \f_{p+1}(T) \leq \t_0(T)$, $\forall\, T\in\cb_\infty^+$.
\end{prop}
\begin{proof}
We first observe that
\begin{gather}\label{Pjn}
\tr(P^j_n)=\#\big\{e\in K_n\colon \length(e)=2^j\big\}=6 \cdot 3^{n-j},\qquad j\leq n.
\end{gather}
Then it is easy to verify that $P^{-p,\infty}\in\cb_0$. Since
\begin{gather}\label{P-pinfty}
\f_p(I)
= \t_0(P^{-p,\infty})
= \lim_{n\to\infty} \frac{\tr P^{-p,n}_n}{\mu_d(K_n)}
= \lim_{n\to\infty} 3^{-n} \sum_{j=-p}^n \tr(P^j_n)
=\sum_{j=-p}^\infty6\cdot 3^{-j}
= 3^{p+2},
\end{gather}
$\f_p$ extends by linearity to a positive functional on $\cb_\infty$. Moreover, by Lemma~\ref{hereditary}, $\f_p(T) \leq \t_0(T)$, $\forall\, T\in\cb_\infty^+$. Finally, since $P^{-p,\infty}P_n=P_nP^{-p,\infty}=P^{-\infty,n}_n$, $\forall\, n\in\bn$, we get, for all $T\in\cb_\infty^+$,
\begin{align*}
\f_{p+1}(T) - \f_p(T)
& = \t_0(P^{-(p+1),\infty}TP^{-(+1)p,\infty}) - \t_0(P^{-p,\infty}TP^{-p,\infty}) \\
&= \lim_{n\to\infty} \frac{\tr((P^{-(p+1),n)}_n -P^{-p,n}_n)T)}{\mu_d(K_n)}
=\lim_{n\to\infty} \frac{\tr(P^{-(p+1))}_n T)}{\mu_d(K_n)}
\geq 0.
\tag*{\qed}
\end{align*}
\renewcommand{\qed}{}
\end{proof}

\begin{prop} \label{def.tau}
Let $\t(T) := \lim\limits_{p\to\infty} \f_p(T)$, $\forall\, T\in\cb_\infty^+$. Then	
\begin{itemize}\itemsep=0pt
\item[$(i)$] $\t$ is a lower semicontinuous weight on $\cb_\infty$,
\item[$(ii)$] $\t(T) = \t_0(T)$, $\forall\, T\in\cai_0^+$.
\end{itemize}
\end{prop}
\begin{proof}
$(i)$ Let $T\in\cb_\infty^+$. Since $\{ \f_p(T) \}_{p\in\bn}$ is an increasing sequence, there exists $\lim\limits_{p\to\infty} \f_p(T) = \sup\limits_{p\in\bn} \f_p(T)$. Then $\t$ is a weight on $\cb_\infty^+$. Since $\f_p$ is continuous, $\t$ is lower semicontinuous.

$(ii)$ Let us prove that, $\forall\, T\in\cb_n^+$,
\begin{gather} \label{uguaglianza}
\frac{\tr(P^j_mT)}{\mu_d(K_m)} = \frac{\tr(P^j_nT)}{\mu_d(K_n)}, \qquad j\leq n \leq m.
\end{gather}
Indeed,
\begin{align*}
\tr\big(P^j_{m+1}T\big) & = \sum_{\substack{e \subset K_{m+1} \\ \length(e)=2^j}} \ps{ e }{ Te } = \sum_{i=0}^2 \sum_{\substack{e \subset C^m_i \\ \length(e)=2^j}} \ps{ e }{ Te } = \sum_{i=0}^2 \sum_{\substack{e \subset K_m \\ \length(e)=2^j}} \ps{ V_{R^m_i}e }{ TV_{R^m_i}e } \\
& = \sum_{i=0}^2 \sum_{\substack{e \subset K_m \\ \length(e)=2^j}} \ps{ e }{ Te } = 3 \tr\big(P^j_mT\big),
\end{align*}
from which~\eqref{uguaglianza} follows. Let us now prove that
\begin{gather} \label{approx}
\t(T) = \sup_{p\in\bn} \f_p(T) = \t_0(T), \qquad T\in\cai_0^+ .
\end{gather}
Let $T\in\cb_n^+ \cap \cai_0^+$, and $\eps>0$. From the definition of $\t_0(T)$, there exists $r\in\bn$, $r>n$, such that $\frac{\tr(P_rT)}{\mu_d(K_r)} > \t_0(T)-\eps$. Since $\frac{\tr(P_rT)}{\mu_d(K_r)} = \sum_{j=-\infty}^r \frac{\tr(P^j_rT)}{\mu_d(K_r)}$, there exists $p\in\bn$ such that $\sum_{j=-p}^r \frac{\tr(P^j_rT)}{\mu_d(K_r)} > \frac{\tr(P_rT)}{\mu_d(K_r)} -\eps > \t_0(T)-2\eps$. Then, for any $s\in\bn$, $s>r$, we have
\begin{align*}
\frac{\tr(P_sP^{-p,\infty}TP^{-p,\infty}P_s)}{\mu_d(K_s)} & = \sum_{j=-p}^s \frac{\tr(P^j_sT)}{\mu_d(K_s)} = \sum_{j=-p}^r \frac{\tr(P^j_sT)}{\mu_d(K_s)} + \sum_{j=r+1}^s \frac{\tr(P^j_sT)}{\mu_d(K_s)} \\
& \stackrel{\eqref{uguaglianza}}{=} \sum_{j=-p}^r \frac{\tr(P^j_rT)}{\mu_d(K_r)} + \sum_{j=r+1}^s \frac{\tr(P^j_sT)}{\mu_d(K_s)} > \t_0(T)-2\eps,
\end{align*}
and, passing to the limit for $s\to\infty$, we get
\begin{gather*}
\f_p(T) = \t_0(P^{-p,\infty}TP^{-p,\infty}) = \lim_{s\to\infty} \frac{\tr(P_sP^{-p,\infty}TP^{-p,\infty}P_s)}{\mu_d(K_s)} \geq \t_0(T)-\eps,
\end{gather*}
and equation~\eqref{approx} follows.
\end{proof}

We want to prove that $\t$ is a tracial weight.

\begin{dfn}%\label{3.2.5}
An operator $U\in B\big(\ell^2(E_\infty)\big)$ is called $\d$-unitary, $\d>0$, if $\|U^*U-1\|<\d$, and $\|UU^*-1\|<\d$.
\end{dfn}

Let us denote with $\cu_\d$ the set of $\d$-unitaries in $\Balg$ and observe that, if $\d<1$, $\cu_\d$ consists of invertible operators, and $U\in\cu_\d$ implies $U^{-1}\in\cu_{\d/(1-\d)}$.

\begin{prop}
The weight $\t_0$ is $\eps$-invariant for $\d$-unitaries in $\Balg$, namely, for any $\eps\in(0,1)$, there is $\d>0$ s.t., for any $U\in\cu_\d$, and $T\in \cb^+_\infty$,
\begin{gather*}
(1-\eps)\t_0(T) \leq \t_0(UTU^*) \leq (1+\eps)\t_0(T) .
\end{gather*}
\end{prop}

\begin{proof}
We first observe that, if $\d\in(0,1)$ and $U\in\cu_\d$, $T\in\cai_0^+\Leftrightarrow UTU^*\in\cai_0^+$. Indeed, choose $n$ such that $U,T\in\cb_n$. Then
$\tr (P_{n} UTU^*) = \tr (U^* UP_{n}TP_{n}) \leq \|U^*U\| \tr (P_{n}T) \leq (1+\d) c_T\vol(K_n)$, $\forall\, n\in\bn$, so that $UTU^* \in \cai_0^+$. Moreover,
\begin{gather*}
\t_0(UTU^*) = \lim_{n\to\infty} \frac{\tr(P_nUTU^*)}{\vol(K_n)} \leq \| U^*U \| \lim_{n\to\infty} \frac{\tr(P_nT)}{\vol(K_n)} = \| U^*U \| \t_0(T) < (1+\d) \t_0(T).
\end{gather*}
Conversely, $UTU^*\in\cai_0^+$, and $U^{-1}\in \cu_{\d/(1-d)} \implies T \in\cai_0^+$. Moreover,
\begin{gather*}
\t_0(T) \leq \big\| \big(U^{-1}\big)^*U^{-1} \big\| \t_0(UTU^*) < \frac1{1-\d} \t_0(UTU^*).
\end{gather*}
The result follows by the choice $\d=\eps$.
\end{proof}

\begin{thm}%\label{3.3.4}
The lower semicontinuous weight $\t$ in Proposition~$\ref{def.tau}$ is a trace on $\cb_\infty$, that is, setting $\cj^+:= \{A\in\cb_\infty^+\colon \t(A)<\infty\}$, and extending $\t$ to the vector space $\cj$ generated by $\cj^+$, we get
\begin{itemize}\itemsep=0pt
\item[$(i)$] $\cj$ is an ideal in $\cb_\infty$,
\item[$(ii)$] $\t(AB)=\t(BA)$, for any $A\in\cj$, $B\in\cb_\infty$.
\end{itemize}
\end{thm}

\begin{proof}
$(i)$ Let us prove that $\cj^+$ is a unitarily-invariant face in $\cb_\infty^+$, and suffices it to prove that $A\in\cj^+$ implies that $UAU^*\in\cj^+$, for any $U\in\cu(\cb_\infty)$, the set of unitaries in $\cb_\infty$. To~reach a~contradiction, assume that there exists $U\in\cu(\cb_\infty)$ such that $\t(UAU^*)=\infty$. Then there is $p\in\bn$ such that $\f_p(UAU^*) > 2\t(A)+2$. Let $\d<3$ be such that $V\in\cu_\d$ implies $\t(VAV^*)\leq 2\t(A)$, and let $U_0\in\Balg$ be such that $\|U-U_0\|< \min\big\{ \frac{\d}{3}, \frac1{3\|A\|\|\f_p\|} \big\}$. The inequalities
\begin{gather*}
\|U_0U_0^*-1\| = \|U^*U_0U_0^*-U^*\| \leq \|U^*U_0-1\|\|U_0^*\|+\|U_0^*-U^*\| < \d
\end{gather*}
and $\|U_0^*U_0-1\|<\d$, prove that $U_0\in\cu_\d$. Since
\[ |\f_p(U_0AU_0^*)-\f_p(UAU^*)|\leq 3\| \f_p \| \|A\| \|U-U_0\| <1,\] we get
\begin{gather*}
2\t(A)\geq \t(U_0AU_0^*) \geq \f_p(U_0AU_0^*) \geq \f_p(UAU^*) - 1 \geq 2\t(A)+1
\end{gather*}
which is absurd.
	
$(ii)$ We only need to prove that $\t$ is unitarily-invariant. Let $A\in\cj^+$, $U\in\cu(\cb_\infty)$. For any $\eps>0$, there is $p\in\bn$ such that $\f_p(UAU^*)>\t(UAU^*)-\eps$, since, by $(1)$, $\t(UAU^*)$ is finite. Then, arguing as in the proof of $(1)$, we can find $U_0\in\Balg$, so close to $U$ that
\begin{gather*}
	 |\f_p(U_0AU_0^*)-\f_p(UAU^*)|<\eps, \\
	 (1-\eps)\t(A)\leq \t(U_0AU_0^*) \leq (1+\eps)\t(A).
\end{gather*}
Then
\begin{align*}
\t(A) & \geq \frac1{1+\eps}\ \t(U_0AU_0^*) \geq \frac1{1+\eps}\ \f_p(U_0AU_0^*)
\geq \frac1{1+\eps}\ (\f_p(UAU^*) -\eps)
\\
& \geq \frac1{1+\eps}\ (\t(UAU^*) -2\eps).
\end{align*}
By the arbitrariness of $\eps>0$, we get $\t(A)\geq \t(UAU^*)$.
Exchanging $A$ with $UAU^*$, we get the thesis.
\end{proof}

\begin{prop}
The lower semicontinuous tracial weight $\t$ defined in Proposition~$\ref{def.tau}$ is semifinite and faithful.
\end{prop}
\begin{proof}
Let us recall that, for any $p\in\bn$, $P^{-p,\infty} \in\cai_0^+$ by Proposition~\ref{def.Qfi}. From Proposition~\ref{def.tau} follows that $\t(P^{-p,\infty})=\t_0(P^{-p,\infty})<\infty$, hence $P^{-p,\infty}\in\cj^+$. Then, for any $T\in\cb_\infty^+$, $S_p := T^{1/2}P^{-p,\infty}T^{1/2} \in\cj^+$, and $0 \leq S_p \leq T$. Moreover,
\begin{align*}
\t(S_p) &= \t\big(T^{1/2}P^{-p,\infty}T^{1/2}\big) = \t(P^{-p,\infty}TP^{-p,\infty}) = \sup_{q\in\bn} \t_0(Q_qP^{-p,\infty}TP^{-p,\infty}Q_q)\\
& = \t_0(P^{-p,\infty}TP^{-p,\infty}) = \f_p(T),
\end{align*}
so that $\sup\limits_{p\in\bn} \t(S_p) = \t(T)$, and $\t$ is semifinite. Finally, if $T\in\cb_\infty^+$ is such that $\t(T)=0$, then $\sup\limits_{p\in\bn} \f_p(T)=0$. Since $\{\f_p(T) \}_{p\in\bn}$ is an increasing sequence, $\f_p(T)=0$, $\forall\, p\in\bn$. Then, for a~fixed $p\in\bn$, we get $0 = \t_0(P^{-p,\infty}TP^{-p,\infty}) = \lim\limits_{n\to\infty} \frac{ \tr(P_nP^{-p,\infty}TP^{-p,\infty}P_n) }{\mu_d(K_n)}$. Since the sequence $\big\{ \frac{ \tr(P_nP^{-p,\infty}TP^{-p,\infty}P_n) }{\mu_d(K_n)} \big\}_{n\in\bn}$ is definitely increasing, we get $\tr(P_nP^{-p,\infty}TP^{-p,\infty}P_n) = 0$ definitely, that is $TP^{-p,\infty}P_n = 0$ definitely, so that $TP^{-p,\infty}=0$. By the arbitrariness of $p\in\bn$, we~get~\mbox{$T=0$}.
\end{proof}

\section[A semifinite spectral triple on the inductive limit A infty]
{A semifinite spectral triple on the inductive limit $\boldsymbol{\ca_\infty}$}\label{sec6}

Since the covering we are studying is ramified, the family $\{\ca_n,\ch_n,D_n\}$ does not have a simple tensor product structure, contrary to what happened in~\cite{AGI01}. We therefore use a different approach to construct a semifinite spectral triple on $\ca_\infty$: our construction is indeed based on the pair $(\cb_\infty,\t)$ of the $C^*$-algebra of geometric operators and the semicontinuous semifinite weight on~it.

The Dirac operator will be defined below (Definition~\ref{Dinfty}) through its phase and the functional calculi of its modulus with continuous functions vanishing at $\infty$.
More precisely we shall use the following

\begin{dfn}
Let $(\mathfrak{C},\tau)$ be a $C^*$-algebra with unit endowed with a semicontinuous semifinite faithful trace. A selfadjoint operator $T$ affiliated to $(\mathfrak{C},\tau)$ is defined as a pair given by a closed subset $\s(T)$ in $\br$ and a $*$ homomorphism $\phi\colon\cc_0(\s(T))\to \mathfrak{C}$, $f(T)\mathop{=}\limits^{\mathrm{def}}\phi(f)$, provided that the support of such homomorphism is the identity in the GNS representation $\pi_\t$ induced by the trace $\t$.
\end{dfn}
The previous definition was inspired by that in~\cite{DFR} appendix A, and should not be confused with that of Woronowicz for $C^*$-algebras without identity.
\begin{rem}\label{varieOps}
The $*$-homomorphism $\phi_\t=\pi_\t\circ\phi$ extends to bounded Borel functions on $\br$ and $e_{(-\infty,t]}\mathop{=}\limits^{\mathrm{def}}\phi_\t(\chi_{(-\infty,t]})$ tends strongly to the identity when $t\to\infty$, hence it is a spectral family. We shall denote by $\pi_\t(T)$ the selfadjoint operator affiliated to $\pi_\t(\mathfrak{C})''$ given by
\begin{gather*}
\pi_\t(T)\mathop{=}\limits^{\mathrm{def}}\int_\br t\, {\rm d} e_{(-\infty,t]}.
\end{gather*}
\end{rem}

\begin{prop}\label{aff-prop}
Let $T$ be a selfadjoint operator affiliated to $(\mathfrak{C},\t)$ as above.
\begin{itemize}\itemsep=0pt
\item[$(a)$] Assume that for any $n\in\bn$, there is $\f_n\in\cc(\br)\colon 0\leq\f_n\leq 1,\f_n=1$ for $|t|\leq a_n$, $\f_n(t)=0$ for $|t|\geq b_n$ with $0<a_n<b_n$ and $\{a_n\}$, $\{b_n\}$ increasing to $\infty$. Then, for any $A\in\mathfrak{C}$, if~$\sup\limits_n\|[T\cdot \f_n(T),A]\|=C<\infty$ then $[\pi_t(T),\pi_\t(A)]$ is bounded and $\|[\pi_t(T),\pi_\t(A)]\|=C$.
\item[$(b)$] If $\t(f(T))<\infty$ for any positive function $f$ with compact support on the spectrum of $T$ then $\pi_\t(T)$ has $\t$-compact resolvent.
\end{itemize}
\end{prop}
\begin{proof} $(a)$ Let $\cd$ be the domain of $\pi_\t(T)$, $\cd_0$ the space of vectors in $\cd$ with bounded support w.r.t.~to $\pi_\t(T)$, and consider the sesquilinear form
$F(y,x)=(\pi_\t(T)y,\pi_\t(A)x)-(y,\pi_\t(A)\pi_\t(T)x)$ defined on $\cd$. By hypothesis, for any $x,y\in\cd_0$ there exists $n$ such that $\pi_\t(\f_n(T))x=x$ and $\pi_\t((\f_n(T))y=y$, hence $F(y,x)=(y,\pi_\t([T\cdot \f_n(T),A])x)\leq C\|x\|\ \|y\|$.
By the density of $\cd_0$ in $\cd$ w.r.t.~the graph norm of $\pi_\t(T)$, the same bound holds on $\cd$.
Then for $y,x\in\cd$, $|(\pi_\t(T)y,\pi_\t(A)x)|\leq |(y,\pi_\t(A)\pi_\t(T)x)|+|F(y,x)|\leq (\|\pi_\t(A)\pi_\t(T)x\|+C\|x\|)\|y\|$ which implies $\pi_\t(A)x$ belongs to the domain of $\pi_\t(T)^*=\pi_\t(T)$. Therefore $\pi_\t(T)\pi_\t(A)-\pi_\t(A)\pi_\t(T)$ is defined on $\cd$ and its norm is bounded by $C$. Since $C$ is the optimal bound for the sesquilinear form $F$ it is indeed the norm of the commutator.

$(b)$ Let $\l$ be in the resolvent of~$|T|$. We then note that for any $f$ positive and zero on a~neigh\-bourhood of the origin there is a $g$ positive and with compact support such that $f\big((|T|-\l I)^{-1}\big)=g(|T|)$.
Therefore $\t\big(f\big((|T|-\l I)^{-1}\big)\big)<\infty$, hence
$\t\big(e_{(t,+\infty)}\big(\pi_\t\big((|T|-\l I)^{-1}\big)\big)\big)<\infty$ for any $t>0$, i.e., $\pi_\t\big((|T|-\l I)^{-1}\big)$ is $\tau$-compact (cf.~Section~\ref{SemST}).
\end{proof}

\begin{dfn}\label{Dinfty}
We consider the Dirac operator $D=F|D|$ on $\ell^2(E_\infty)$, where $F$ is the orientation reversing operator on edges and
\begin{gather*}
|D|=\sum_{n\in\bz}2^{-n}P^n,\qquad \s(|D|)=\{2^{-n},\ n\in\bz\}\cup\{0\}.
\end{gather*}
\end{dfn}

\begin{prop}\label{OnDinfty}
The following hold:
\begin{itemize}\itemsep=0pt
\item[$(a)$] The
elements $D$ and $|D|$ are affiliated to $(\cb_\infty,\tau)$.
\item[$(b)$] The following formulas hold: $\t(P^n)=6\cdot 3^{-n}$, $\t(P^{-p,\infty})=3^{p+2}$, as a consequence the operator $D$ has $\t$-compact resolvents
\item[$(c)$] The trace $\t(I+ D^{2})^{-s/2}<\infty$ if and only if $s>d=\frac{\log3}{\log2}$ and
\begin{gather*}
\Res_{s=d}\t\big(I+ D^{2}\big)^{-s/2}=\frac6{\log2}.
\end{gather*}
\end{itemize}
\end{prop}
\begin{proof}
($a$) We first observe that the $*$-homomorphisms for $D$ and $|D|$ have the same support projection, then note that since $F$ and $P_n$ belong to $\cb_0$ (which is a von Neumann algebra) for any $n\in\bn$, then $f(D)$ and $f(|D|)$ belong to $\cb_0$ for any $f\in\cc_0(\br)$; therefore it is enough to show that the support of $f\mapsto f( |D|)$ is the identity in the representation $\pi_\t$.

In order to prove this, it is enough to show that $\pi_\t(e_{|D|} [0,2^p])$ tends to the identity strongly when $p\to\infty$, that is to say that $\pi_\t(e_{|D|}(2^p,\infty))$ tends to 0 strongly when $p\to\infty$.

We consider then the projection $P^{-\infty,0}$ which projects on the space generated by the edges with $\length(e)\leq 1$. Clearly, such projection belongs to $\cb_0$, we now show that it is indeed central there. In fact, if $c$ is a cell with $\size(c)=1$, $P_c$ commutes with $\cb_0$. Since $P^{-\infty,0}=\sum_{\size(c)=1}P_c$, then $P^{-\infty,0}$ commutes with $\cb_0$.
On the one hand, the von Neumann algebra $P^{-\infty,0}\cb_0$ is isomorphic to $\cb\big(\ell^2(K)\big)$ and the restriction of $\t$ to $P^{-\infty,0}\cb_0$ coincides with the usual trace on $\cb\big(\ell^2(K)\big)$, therefore the representation $\p_\t$ is normal when restricted to $P^{-\infty,0}\cb_0$. On the other hand, since $e_{|D|}(2^p,\infty)=P^{-\infty,-p-1}$ is, for $-p\leq 1$, a sub-projection of $P^{-\infty,0}$, and $P^{-\infty,-p-1}$ tends to 0 strongly in the given representation, the same holds of the representation $\pi_\t$.

($b$) We prove the first equation. Indeed
\begin{gather*}
\t(P^n)
=\lim_m \frac{\tr P_{m}^n}{\vol(K_m)}
=\tr P_{0}^n+\lim_m\sum_{j=1}^m\frac{\tr P_j^j P^n}{\vol(K_j)}.
\end{gather*}
The first summand is non-zero iff $n\leq0$, while the second vanishes exactly for such $n$. Since
\begin{gather*}
\lim_m\sum_{j=1}^m\frac{\tr P_j^j P^n}{\vol(K_j)}
=\frac{\tr P_n^n }{\vol(K_n)},
\end{gather*}
the result in~\eqref{Pjn} shows that in both cases we obtain $6\cdot 3^{-n}$. We already proved in~\eqref{P-pinfty} that $\t_0(P^{-p,\infty})=3^{p+2}$. Since $P^{-p,\infty}\in\cb_0$, the same holds for $\t$ by Proposition~\ref{def.tau}$(ii)$.
Then the thesis follows by condition $(b)$ in Proposition~\ref{aff-prop}.

($c$) We have $\t\big(I+ D^{2}\big)^{-s/2}=\t\big(P^{-\infty,0}\big(I+ D^{2}\big)^{-s/2}\big)
+\t\big(P^{1,+\infty}\big(I+ D^{2}\big)^{-s/2}\big)$.
A straightforward computation and~\eqref{evConst} give
\begin{gather*}
\t\big(P^{-\infty,0}\big(I+ D^{2}\big)^{-s/2}\big)=\tr\big(P_0\big(I+ D^{2}\big)^{-s/2}\big)=
6\sum_{n\geq0}\big(1+2^{2n}\big)^{-s/2}3^{n},
\end{gather*}
which converges iff $s>d$. As for the second summand, we have
\begin{align*}
\t\big(P^{1,+\infty}\big(I+ D^{2}\big)^{-s/2}\big)
&=\t_0\big(P^{1,+\infty}\big(I+ D^{2}\big)^{-s/2}\big)
=\lim_m\frac{\tr\big(P^{1,m}_m\big(I+ D^{2}\big)^{-s/2}\big)}{\mu_d(K_m)}
\\
&=\lim_m\sum_{j=1}^m 3^{-m}\tr\big(P^{1,m}_m\big(I+ D^{2}\big)^{-s/2}\big)
=6\sum_{j=1}^\infty 3^{-j}\big(1+2^{-2j}\big)^{-s/2},
\end{align*}
which converges for any $s$ hence does not contribute to the residue. Finally
\begin{align*}
\Res_{s=d}\t\big(I+ D^{2}\big)^{-s/2}
&=\lim_{s\to d^+}(s-d)\t\big(I+ D^{2}\big)^{-s/2}\\
&=\lim_{s\to d^+}\bigg(s-\frac{\log3}{\log2}\bigg)6\sum_{n\geq0}\big(1+2^{-2n}\big)^{-s/2}
e^{n(\log3-s\log2)}\\
&=\frac6{\log2}\lim_{s\to d^+}\frac{s\log2-\log3}{1-e^{-(s\log2-\log3)}}
=\frac6{\log2}. \tag*{\qed}
\end{align*}
\renewcommand{\qed}{}
\end{proof}

\begin{prop}\label{commutator}
For any $f\in\ca_n$ $\sup\limits_{t>0}\big\|\big[ e_{[-t,t]}(D)\, D,\rho(f)\big]\big\|
=\|[ D_n,\rho(f|_{K_n})]\|$.
\end{prop}
\begin{proof}
We observe that
$| D|$ is a multiplication operator on $\ell^2(E_\infty)$, therefore it commutes with~$\rho(f)$. Hence,
\begin{gather*}
\big\|\big[ D\,e_{[-2^p,2^p]}(D),\rho(f)\big]\big\|
=\big\||D|\,e_{[0,2^p]}(|D|)\, (\rho(f)-F\rho(f)F)\big\|
\!=\!\!\sup_{\length(e)\geq 2^{-p}}\!\!\!\frac{|f(e^+)-f(e^-)|}{\length(e)}.
\end{gather*}
As a consequence,
\begin{gather*}
\sup_{p\in\bz}\big\|\big[ D\,e_{[-2^p,2^p]}(D),\rho(f)\big]\big\|
=\sup_{e\in E_\infty}\frac{|f(e^+)-f(e^-)|}{\length(e)}.
\end{gather*}
Recall now that any edge $e$ of length $2^{n+1}$ is the union of two adjacent edges $e_1$ and $e_2$ of~length~$2^n$ such that $e_1^+=e_2^-$, therefore
\begin{gather*}
\frac{|f(e^+)-f(e^-)|}{2^{n+1}}
\leq\frac12\bigg(\frac{|f(e_1^+)-f(e_1^-)|}{2^{n}}+\frac{|f(e_2^+)-f(e_2^-)|}{2^{n}}\bigg)
\leq \sup_{\length(e)=2^n}\frac{|f(e^+)-f(e^-)|}{\length(e)}.
\end{gather*}
Iterating, we get
\begin{gather*}
\sup_{e\in E_\infty}\frac{|f(e^+)-f(e^-)|}{\length(e)}
=\sup_{\length(e)\leq 2^{n}}\frac{|f(e^+)-f(e^-)|}{\length(e)}.
\end{gather*}
Since $f\in\ca_n$,
\begin{gather*}
\sup_{\length(e)\leq 2^{n}}\frac{|f(e^+)-f(e^-)|}{\length(e)}
=\sup_{e\in K_n}\frac{|f(e^+)-f(e^-)|}{\length(e)}
=\|[ D_n,\rho(f|_{K_n})]\|.
\tag*{\qed}
\end{gather*}
\renewcommand{\qed}{}
\end{proof}

In the following Theorem we identify $\cb_\infty$ with $\pi_\t(\cb_\infty)$, the trace $\t$ on $\pi_\t(\cb_\infty)$ with its exten\-sion to $\pi_\t(\cb_\infty)''$, and $D_n$ and $D$ as unbounded operators affiliated with $(\cb_\infty,\t)$ with~$\pi_\t(D_n)$ and $\pi_\t(D)$ as unbounded operators affiliated with $(\pi_\t(\cb_\infty)'',\t)$.

\begin{thm}\label{SFtriple}
The triple $(\cl,\pi_\t(B_\infty)'',D)$ on the unital C$^*$-algebra $\ca_\infty$ is an odd semifinite spectral triple, where $\cl=\cup_n\{f\in\ca_n, f\text{\,Lipschitz}\}$.
The spectral triple has metric dimension $d=\frac{\log3}{\log2}$, the functional
\begin{gather}\label{int-trace}
\oint f=\t_\om \big(\rho(f)\big(I+D^2\big)^{-d/2}\big),
\end{gather}
 is a finite trace on $\ca_\infty$ where $\t_\om$ is the logarithmic Dixmier trace associated with $\t$, and
\begin{gather}\label{integral}
\oint f=\frac6{\log3}\frac{\int_{K_n} f\,{\rm d}\vol}{\vol(K_n)},\qquad f\in\ca_n,
\end{gather}
where $\vol$ is the Hausdorff measure of dimension $d$ normalized as above. As a consequence, $\oint f$ is a Bohr--F{\o}lner mean on the solenoid:
\begin{gather*}
\oint f=\frac6{\log3}\,\lim_{n\in\bn}\frac{\int_{K_n} f\,{\rm d}\vol}{\vol(K_n)},\qquad
f\in\ca_\infty.
\end{gather*}
The Connes distance
\begin{gather*}
d(\f,\psi)=\sup\{|\f(f)-\psi(f)|\colon f\in\cl,\, \| [ D,\rho(f) ] \| = 1 \},\qquad
\f,\psi\in\cs(\ca_\infty)
\end{gather*}
between states on $\ca_\infty$ verifies
\begin{gather}\label{distance}
d(\d_x,\d_y)=d_{\rm geo}(x,y),\qquad x,y\in K_\infty,
\end{gather}
where $\dg$ is the geodesic distance on $K_\infty$.
\end{thm}

\begin{proof}
The properties of a semifinite spectral triple follow by the properties proved above, in particular property $(1)$ of Definition~\ref{def:SFtriple} follows by Propositions~\ref{aff-prop}$(a)$ and~\ref{commutator}, while pro\-perty $(2)$ follows by Proposition~\ref{OnDinfty}$(b)$.
The functional in equality~\eqref{int-trace} is a finite trace by Proposition~\ref{OnDinfty}$(c)$. Equations~\eqref{integral} and~\eqref{distance} only remain to be proved. We observe that $(I+D^2)^{-d/2}-|D_n|^{-d}$ have finite trace. Indeed
\begin{gather*}
\big((I+D^2)^{-d/2}-|D_n|^{-d}\big)e=
\begin{cases}
\big(1+4^{-k}\big)^{-d/2}e, & \length(e)=2^k,\quad k>n,
\\
\big(\big(1+4^{-k}\big)^{-d/2}-2^{dk}\big)e, & \length(e)=2^k,\quad k>n k\leq n,
\end{cases}
\end{gather*}
hence, makig use of a formula in Theorem~\ref{OnDinfty}$(b)$, we get
\begin{align*}
\big|\t\big(\big(I+D^2\big)^{-d/2}-|D_n|^{-d}\big)\big|
&\leq\sum_{k>n}(1+4^{-k})^{-d/2}\t\big(P^k\big)+\sum_{k\leq n}\big|\big(1+4^{-k}\big)^{-d/2}-3^k\big|\t\big(P^k\big)\\
&\leq6\big(1+4^{-(n+1)}\big)^{-d/2}\sum_{k>n}3^{-k}
+6\sum_{k\leq n}\big|\big(1+4^{k}\big)^{-d/2}-1\big|
\end{align*}
and both series are convergent. Since the Dixmier trace vanishes on trace class operators, this implies that
\begin{gather*}
\t_\om (\rho(f)\big(I+D^2\big)^{-d/2})=\t_\om \big(\rho(f)|D_n|^{-d}\big)
=\frac1d \Res_{s=d}\t\big(\rho(f)|D_n|^{-s}\big),
\end{gather*}
therefore, if $f\in\ca_n$,
\begin{gather*}
\oint f
=\frac1d \Res_{s=d}\t\big(\rho(f)|D_n|^{-s}\big)
=\frac1d \Res_{s=d}\frac{\tr(\rho(f_{K_n})|D_n|^{-s})}{\vol(K_n)}
=\frac{\tr_\om \big(\rho(f)|D_n|^{-d}\big)}{\vol(K_n)}.
\end{gather*}
Now, by formula~\eqref{fractalNCint} applied to $K_n$,
$\tr_\omega(\rho(f)|D_n|^{-d}) = \frac{6\cdot \ell(e)^d}{ \log 3} \int_{K_n} f\, {\rm d}H_d$, where $H_d$ is the Hausdorff measure normalized on $K_n$, hence $H_d=(\m_d(K_d))^{-1}\m_d=3^{-n}\m_d$, and $e\in E_0(K_n)$, hence $\ell(e)^d=3^n$. Therefore
$\tr_\omega\big(\rho(f)|D_n|^{-d}\big) = \frac6{\log3} \int_{K_n} f\, {\rm d}\m_d$ and formula~\eqref{integral} follows.
As~for equation~\eqref{distance}, given $x,y\in K_\infty$ let $n$ such that $x,y\in K_n$, $m\geq n$. Then, combining Propositions~\ref{aff-prop}$(a)$ and~\ref{commutator}, we have, for $f\in\ca_m$,
\begin{gather*}
\|[ D, \rho(f)]\|=\|[ D_m,\rho(f|_{K_m})]\|,
\end{gather*}
and, by Theorem 5.2 and Corollary 5.14 in~\cite{GuIs16},
\begin{gather*}
\sup\{|f(x)-f(y)|\colon f\in\ca_m,\|[D_m,\rho(f|_{K_n})]\|=1\}=\dg(x,y),\qquad m\geq n.
\end{gather*}
Therefore
\begin{align*}
d(\d_x,\d_y)
&=\sup\{|f(x)-f(y)|\colon f\in\cl,\, \|[ D,\rho(f)]\|=1\}
\\
&=\lim_m\ \sup\{|f(x)-f(y)|\colon f\in\ca_m,\, \|[ D,\rho(f)]\|=1\}
\\
&=\lim_m\ \sup\{|f(x)-f(y)|\colon f\in\ca_m,\, \|[D_m,\rho(f|_{K_n})]\|=1\}=\dg(x,y).
\tag*{\qed}
\end{align*}
\renewcommand{\qed}{}
\end{proof}
\begin{rem}
The last statement in Theorem~\ref{SFtriple} shows that the triple $(\cl,\cam,D_\infty)$ recovers two incompatible aspects of the space $\ca_\infty$: on the one hand the compact space given by the spectrum of the unital algebra $\ca_\infty$, with the corresponding finite integral, and on the other hand the open fractafold $K_\infty$ with its geodesic distance. In particular, the functional on $\cl$ given by~$L(f)=\|[D,\rho(f)]\|$ is not a Lip-norm in the sense of Rieffel~\cite{Rieffel} because it does not give rise to~the weak$^*$ topology on $\cs(\ca_\infty)$. In fact, such seminorm produces a distance which is unbounded on points, therefore the induced topology cannot be compact.
\end{rem}

\subsection*{Acknowledgements}
We thank the referees of this paper for many interesting observations and suggestions.
V.A.~is supported by the Swiss National Science Foundation.
D.G.~and T.I.~are supported in part by~GNAMPA-INdAM and the ERC Advanced Grant 669240 QUEST ``Quantum Algebraic Structures and Models'', and acknowledge the MIUR Excellence Department Project awarded to the Department of Mathematics, University of Rome Tor Vergata, CUP E83C18000100006.

\pdfbookmark[1]{References}{ref}
\LastPageEnding

\end{document}